\documentclass{article}

\usepackage{arxiv}

\usepackage[utf8]{inputenc} 
\usepackage[T1]{fontenc}    
\usepackage{hyperref}       
\usepackage{url}            
\usepackage{booktabs}       
\usepackage{amsfonts}       
\usepackage{nicefrac}       
\usepackage{microtype}      
\usepackage{graphicx}
\usepackage{natbib}
\usepackage{doi}
\usepackage{todonotes}
\usepackage{setspace}
\usepackage{thmtools}

\usepackage[export]{adjustbox}
\usepackage{caption}

\newcommand{\squeeze}{}

\usepackage{multirow}
\usepackage{layout}

\usepackage{amsmath}
\usepackage{amsthm}
\usepackage{amssymb}
\usepackage{mathtools}
\usepackage{siunitx}

\usepackage{enumitem}

\usepackage{bbm}

\usepackage{color}
\usepackage{verbatim}
\usepackage{subcaption}

\usepackage{cleveref}

\usepackage{apptools}
\usepackage{thmtools}
\usepackage{algorithm, algpseudocode}

\usepackage[flushleft]{threeparttable}
\usepackage{array,makecell}

\newcommand{\norm}[1]{\left\| #1 \right\|}

\newcommand{\inp}[2]{\left\langle#1,#2\right\rangle} 

\newcommand{\parens}[1]{\left( #1 \right)}
\newcommand{\brac}[1]{\left\{ #1 \right\}}

\newcommand{\cB}{\mathcal{B}}
\newcommand{\cC}{\mathcal{C}}

\newcommand{\cN}{\mathcal{N}}
\newcommand{\cO}{\mathcal{O}}

\newcommand{\cS}{\mathcal{S}}

\newcommand{\cX}{\mathcal{X}}

\newcommand{\mA}{\mathbf{A}}

\newcommand{\mG}{\mathbf{G}}

\newcommand{\mI}{\mathbf{I}}

\newcommand{\mP}{\mathbf{P}}

\newcommand{\mZ}{\mathbf{Z}}
\newcommand{\mU}{\mathbf{U}}

\newcommand{\mV}{\mathbf{V}}
\newcommand{\mY}{\mathbf{Y}}
\newcommand{\mX}{\mathbf{X}}

\newcommand{\ceil}[1]{\left\lceil #1\right\rceil}

\newcommand{\R}{\mathbb{R}} 

\newcommand{\eqdef}{:=} 

\DeclareMathOperator{\vol}{vol}               

\DeclareMathOperator{\sign}{sign}

\DeclareMathOperator{\dom}{dom}         

\DeclareMathOperator{\interior}{int}    
\DeclareMathOperator{\ri}{ri}           
\DeclareMathOperator{\sri}{sri}         
\DeclareMathOperator{\bdry}{bdry}       

\DeclareMathOperator{\diag}{diag}       
\DeclareMathOperator*{\argmin}{arg\,min}
\DeclareMathOperator*{\argmax}{arg\,max}

\usepackage{tcolorbox}
\usepackage{pifont}
\definecolor{myred}{RGB}{238,102,119}
\definecolor{myred2}{RGB}{215,60,50}
\definecolor{myblue}{RGB}{68,119,170}
\definecolor{yaleblue}{rgb}{0.06, 0.3, 0.57}
\definecolor{mydarkgreen}{RGB}{39,130,67}
\definecolor{mydarkred}{RGB}{192,47,25}
\definecolor{mydarkorange}{RGB}{39,130,67}

\newcommand{\algnamesmall}[1]{{\small \sf #1}}
\newcommand{\algnamefootnote}[1]{{\footnotesize \sf #1}}

\newtheorem{lemma}{Lemma}

\newtheorem{theorem}{Theorem}

\newtheorem{corollary}{Corollary}
\newtheorem{fact}{Fact}

\theoremstyle{plain}

\newtheorem{remark}[theorem]{Remark}

\theoremstyle{definition}

\newtheorem{definition}[theorem]{Definition}

\usepackage{listings}
\usepackage{xcolor}

\definecolor{codegreen}{rgb}{0,0.6,0}
\definecolor{codegray}{rgb}{0.5,0.5,0.5}
\definecolor{codepurple}{rgb}{0.58,0,0.82}
\definecolor{backcolour}{rgb}{0.95,0.95,0.92}

\lstdefinestyle{mystyle}{
    backgroundcolor=\color{backcolour},   
    commentstyle=\color{codegreen},
    keywordstyle=\color{magenta},
    numberstyle=\tiny\color{codegray},
    stringstyle=\color{codepurple},
    basicstyle=\ttfamily\footnotesize,
    breakatwhitespace=false,         
    breaklines=true,                 
    captionpos=b,                    
    keepspaces=true,                 
    numbers=left,                    
    numbersep=5pt,                  
    showspaces=false,                
    showstringspaces=false,
    showtabs=false,                  
    tabsize=2
}

\newcommand{\rbrac}[1]{\left(#1\right)}

\newcommand{\cbrac}[1]{\left\{#1\right\}}

\newcommand{\BProxSub}[3]{\textnormal{{brox}}^{#1}_{#2}(#3)}
\newcommand{\ProxSub}[2]{\textnormal{{prox}}_{#1}(#2)}

\newcommand{\lmo}[2]{{\rm LMO}_{#1}(#2)}

\allowdisplaybreaks

\title{Non-Euclidean Broximal Point Method:\\ A Blueprint for Geometry-Aware Optimization}

\author{Kaja Gruntkowska \\
	KAUST\thanks{King Abdullah University of Science and Technology}\\
    	Center of Excellence for Generative AI\\ 	
	 Thuwal, Saudi Arabia \\
    	\And
    	Peter Richt\'{a}rik \\
    	KAUST$^*$ \\	
    	Center of Excellence for Generative AI\\ 
        Thuwal, Saudi Arabia \\
}

\date{}

\renewcommand{\shorttitle}{Non-Euclidean Broximal Point Method: A Blueprint for Geometry-Aware Optimization}

\usepackage[hyperpageref]{backref}
\hypersetup{ 
    colorlinks = true,
    urlcolor = yaleblue,
    linkcolor = yaleblue,
    citecolor = yaleblue
}
\renewcommand*{\backref}[1]{}
\renewcommand*{\backrefalt}[4]{%
   \ifcase #1 %
     \footnotesize{(Not cited.)}%
   \or
     \footnotesize{(Cited on page~#2)}%
   \else
     \footnotesize{(Cited on page~#2)}%
\fi }

\begin{document}
\maketitle

\begin{abstract}
    The recently proposed Broximal Point Method (\algnamesmall{BPM}) \citep{gruntkowska2025ball} offers an idealized optimization framework based on iteratively minimizing the objective function over norm balls centered at the current iterate. It enjoys striking global convergence guarantees, converging linearly and in a finite number of steps for proper, closed and convex functions. However, its theoretical analysis has so far been confined to the \emph{Euclidean} geometry. At the same time, emerging trends in deep learning optimization, exemplified by algorithms such as \algnamesmall{Muon} \citep{jordan2024muon} and \algnamesmall{Scion} \citep{pethick2025training}, demonstrate the practical advantages of minimizing over balls defined via \emph{non-Euclidean} norms which better align with the underlying geometry of the associated loss landscapes.
    In this note, we ask whether the convergence theory of \algnamesmall{BPM} can be extended to this more general, non-Euclidean setting. We give a positive answer, showing that most of the elegant guarantees of the original method carry over to arbitrary norm geometries. Along the way, we clarify which properties are preserved and which necessarily break down when leaving the Euclidean realm. Our analysis positions Non-Euclidean \algnamesmall{BPM} as a conceptual blueprint for understanding a broad class of geometry-aware optimization algorithms, shedding light on the principles behind their practical effectiveness.
\end{abstract}

\section{Introduction}

Optimization stands as a cornerstone of modern machine learning, powering the success of virtually everything, from deep neural networks to state-of-the-art large language models. As models grow ever larger and more complex, the demand for better optimization algorithms evolves in tandem. What is needed now is a new class of methods--ones that are not only computationally efficient and scalable, but also inherently robust to the nonconvex, high-dimensional landscapes that characterize contemporary deep learning.

The field of deep learning optimization has long relied on \algnamesmall{Adam} and related algorithms \citep{kingma2014adam, loshchilov2017decoupled}, whose core innovation lies in adaptive moment estimation. While these methods have achieved significant empirical success, their theoretical behavior, particularly in nonconvex landscapes, remains only partially understood.
Recently, however, a new class of optimization methods has begun to challenge this long-standing dominance. Algorithms such as \algnamesmall{Muon} \citep{jordan2024muon}, \algnamesmall{Scion} \citep{pethick2025training} and  \algnamesmall{Gluon} \citep{riabinin2025gluon} break away from the adaptive moment paradigm. Instead, they adopt a different principle: structured updates derived by minimizing a linear approximation of the loss function over a carefully chosen norm ball. Crucially, these norm balls are \emph{non-Euclidean}, aiming to reflect the intrinsic geometry of the optimization problem (see \Cref{sec:muon_intro}). Operating at the individual layer level, these algorithms combine simplicity, scalability, and promising empirical performance, with several studies suggesting their potential to outperform \algnamesmall{Adam(W)} in large-scale deep learning tasks \citep{liu2025muon, pethick2025training, shah2025practical, therien2025muloco, tveit2025muon}.

While \algnamesmall{Muon} and \algnamesmall{Scion}---both of which utilize spectral norm balls---are among the most prominent and actively studied examples in this emerging line of work, the underlying principle extends well beyond these specific instances. By changing the geometry through altering the norm, one can recover a range of familiar optimization methods. For example, selecting $\ell_1$, $\ell_2$, or $\ell_\infty$ norms yields Coordinate Descent (\algnamesmall{CD}) \citep{luo1992CD,nesterov2012efficiency,richtarik2011iteration,wright2015coordinate,sigma_k}, normalized Gradient Descent (\algnamesmall{$\|$GD$\|$}) \citep{nesterov1984minimization, hazan2015beyond, Momentum_Norm_SGD_2020, Orabona_normalized_GD, EF-norm}, and Sign Gradient Descent (\algnamesmall{SignGD}) \citep{riedmiller1993direct, bernstein2018signsgd, safaryan2021stochastic}, respectively. We explore these connections in more detail in Sections \ref{sec:muon_intro},~\ref{sec:lit_rev}, and \Cref{sec:lin_bpm}.

The structural simplicity of this family of methods has sparked renewed interest in understanding the interplay between the geometric properties of the optimization landscape and the resulting theoretical and empirical performance \citep{kovalev2025understanding, li2025note, pethick2025training, riabinin2025gluon}. Still, despite these advances, the theoretical groundwork is far from complete. We are merely at the initial stages of piecing together the mathematical explanations for their behavior, success, and deeper connections to core optimization concepts.

One notable effort in this direction is the Broximal Point Method (\ref{eq:bpm_e}), recently proposed by \citet{gruntkowska2025ball}--a theoretically grounded algorithm that captures one of the core principles of this emerging class of optimizers. While \algnamesmall{Muon}-style methods iteratively minimize surrogate linear models over ball-constrained regions, \algnamesmall{BPM} takes a more direct route by targeting the original objective function itself. This distinction enables \algnamesmall{BPM} to enjoy remarkable convergence properties, supported by a clean and elegant theoretical analysis (see \Cref{thm:bpm_e}).

Yet a key discrepancy remains: unlike \algnamesmall{Muon} and \algnamesmall{Scion}, which operate over \emph{non-Euclidean} norm balls tailored to the problem geometry, \algnamesmall{BPM} is confined to the \emph{Euclidean} setting. This raises a natural question:
\begin{center}
\emph{Can the theoretical benefits of \algnamesmall{BPM} be extended to non-Euclidean geometries?}
\end{center}
In this work, we give a partially affirmative answer to this question. By extending the convergence theory of \citet{gruntkowska2025ball} to general norm settings, we take a step toward aligning the theoretical foundations of \algnamesmall{BPM} with the algorithms actually employed in modern machine learning practice.

\subsection{Outline}

The structure of this work is as follows. In \Cref{sec:euclidean}, we introduce the general setup and walk the reader through the main motivation behind our work--the Broximal Point Method \citep{gruntkowska2025ball}. We provide a comprehensive introduction to \algnamesmall{BPM}, outline its key convergence properties, and offer additional commentary on these results, including some insights not discussed in the original paper.
\Cref{sec:non_euclidean} is devoted to the non-Euclidean extension of the method. We begin by motivating the transition beyond Euclidean geometry in \Cref{sec:muon_intro}, and then present our main contribution: Non-Euclidean \algnamesmall{BPM}, introduced in \Cref{sec:non_euclidean_bpm}. The theoretical guarantees of the method are established in \Cref{thm:conv-nbpm}.
In \Cref{sec:lit_rev}, we provide a broader review of related literature and highlight the various areas of optimization that our work connects with, including ball oracles, linear minimization oracles, preconditioning techniques, and trust region methods. We conclude the paper with a summary of our findings (\Cref{sec:conclusion}).

\section{Conceptual Foundations}\label{sec:euclidean}

Motivated by the empirical success of optimization methods that iteratively minimize (models of) the loss over non-Euclidean balls, this work seeks to extend the convergence theory of the \algnamesmall{BPM} algorithm beyond the Euclidean case.
To set the stage, let $\cS$ be a finite-dimensional vector space equipped with an inner product $\inp{\cdot}{\cdot}: \cS \times \cS \to \R$, which induces the standard Euclidean norm $\norm{\cdot}_2$. We additionally endow $\cS$ with a potentially non-Euclidean norm $\norm{\cdot}: \cS \to \R_{\geq0}$, and denote its dual norm by $\norm{\cdot}_{\star}: \cS \to \R_{\geq0}$, defined via $\norm{x}_{\star} \eqdef \sup_{\norm{z}\leq 1} \inp{x}{z}$.

Within this setting, we study the general optimization problem
\begin{align}\label{eq:problem}
    \min_{x \in \cS} f(x),
\end{align}
where $f: \cS \mapsto \R \cup \{+\infty\}$ is a proper (that is, ${\rm dom} f \eqdef \{x \in \cS : f(x)<+\infty\}$ is non-empty), closed and convex function with at least one minimizer. We denote the set of minimizers of $f$ by $\cX_{\star}$ and the optimal value by $f_\star \eqdef \inf_{x \in \cS} f(x)$.
This broad formulation encompasses a vast array of problems in optimization, machine learning, signal processing, computational biology, and applied mathematics. Our focus in the remainder of the paper is to understand how the non-Euclidean geometry---encoded via $\norm{\cdot}$---affects the theoretical guarantees of \algnamesmall{BPM} when the Euclidean norm ball constraint is replaced by one induced by $\norm{\cdot}$.

Throughout the paper, lowercase letters (e.g., $x$) denote vectors in $\cS$, while bold uppercase letters (e.g.,~$\mX$) represent matrices. We write $\boldsymbol{0}$ and $\mI$ for the zero and identity matrices, respectively, with dimensions clear from the context. For any $\sigma \in \R^d$, the notation $\diag(\sigma)$ refers to the diagonal matrix in $\R^{d \times d}$ whose diagonal entries are given by $\sigma$.

\subsection{From Proximal to Broximal Point Method}

Before detailing our contributions, let us retrace the path that led to this work. The starting point lies in the fundamental challenge of global nonconvex optimization--one that is, in general, NP hard \citep{murty1987someNP}. Motivated by the desire to tackle such problems, \citet{gruntkowska2025ball} propose an idealized meta-algorithm, the Broximal Point Method (\algnamesmall{BPM}). Though abstract, this method offers a fresh, theoretically grounded perspective on optimizer design.

The key inspiration behind \algnamesmall{BPM} is the classical Proximal Point Method (\algnamesmall{PPM}) \citep{rockafellar1976monotone}, which augments the objective function with a quadratic penalty term and iteratively solves the resulting regularized subproblems. Formally, consider problem \eqref{eq:problem} with $\cS = \R^d$. \algnamesmall{PPM} aims to solve it via the update rule
\begin{align*}
    x_{k+1} = \ProxSub{\gamma_k f}{x_k} \eqdef \argmin\limits_{z\in\R^d} \brac{f(z) + {\color{myred} \frac{1}{2 \gamma_k} \norm{z-x_k}_2^2}}, \tag{\algnamesmall{PPM}}
\end{align*}
where $\gamma_k>0$ is a stepsize and $\ProxSub{\gamma f}{x} \eqdef \argmin_{z\in\R^d} \brac{f(z) + \frac{1}{2 \gamma} \norm{z-x}_2^2}$ denotes the \emph{proximal operator}.
Building on this idea, \citet{gruntkowska2025ball} replace the {\color{myred} quadratic penalty} term with a hard {\color{myblue} ball constraint}. The resulting method, \algnamesmall{BPM}, performs the update
\begin{align}\label{eq:bpm_e}
    x_{k+1} = \BProxSub{t_k}{f}{x_k} \eqdef \argmin\limits_{z\in\R^d} \brac{f(z): {\color{myblue} \norm{z - x_k}_2 \leq t_k}}. \tag{\algnamesmall{BPM}}
\end{align}
That is, at each iteration, \algnamesmall{BPM} applies the \emph{Ball-proximal (``broximal'') operator} $\BProxSub{t}{f}{x} \eqdef \argmin_{z \in \cB_2(x,t)} f(z)$, moving from $x_k$ to a minimizer of $f$ within the \emph{Euclidean} ball $\cB_2(x_k, t_k) \eqdef \brac{z \in \R^d : \norm{z - x_k}_2 \leq t_k}$ of radius~$t_k$ centered at~$x_k$.

Surprisingly, this modification yields a host of elegant convergence properties, many of which go far beyond what is typically achievable under assumptions as minimal as those of \Cref{thm:bpm_e} stated below. Although the original \algnamesmall{BPM} paper also treats the nonconvex setting---motivated by the goal of developing methods capable of escaping local minima (a task that \algnamesmall{BPM} can accomplish whenever the radius $t_k$ is sufficiently large)---for clarity we restrict our attention to the convex case to highlight the core theoretical guarantees of \algnamesmall{BPM}.

\begin{theorem}[\citet{gruntkowska2025ball}, Theorem 8.1]\label{thm:bpm_e}
    Assume that $f: \R^d \mapsto \R \cup \{+\infty\}$ is proper, closed and convex, $\cX_\star \neq \emptyset$, and let $\{x_k\}_{k\geq0}$ be the iterates of \ref{eq:bpm_e} run with any sequence of positive radii $\{t_k\}_{k\geq0}$, where $x_0\in {\rm dom} f$. Then

    \begin{enumerate}
        \item \textbf{One Step Convergence.}
        If $\cX_{\star}\cap \cB_2(x_k, t_k)\neq\emptyset$, then $x_{k+1}\in\cX_{\star}$, i.e., $x_{k+1}$ is optimal.
        
        \textnormal{\color{gray} This holds regardless of the radius size, implying that \algnamesmall{BPM} can converge \emph{arbitrarily fast}, even in a single iteration, if the radius~$t_0$ is large enough, i.e., $t_0 \geq \textnormal{dist}(x_0, \cX_{\star})$, where $\textnormal{dist}(x, \cC) \eqdef \inf_{z\in\cC} \norm{x-z}_2$.}
    
        \item \textbf{Super-Accelerated Linear Convergence in Distance to Minimizer.}
        If $\cX_{\star}\cap \cB_2(x_k, t_k)=\emptyset$, then $x_{k+1}$ is a singleton and $\norm{x_{k+1} - x_k}_2 = t_k$, meaning that the iterates move from the center to the boundary of the ball $\cB_2(x_k, t_k)$ (hence, $t_k$ can be thought of as the effective stepsize). Moreover, for any $x_{\star}\in\cX_{\star}$,
        \begin{align}\label{eq:d_conv_e}
            \norm{x_{k+1} - x_{\star}}_2^2 \leq \norm{x_k - x_{\star}}^2_2 - t_k^2.
        \end{align}
        \textnormal{\color{gray} Inequality \eqref{eq:d_conv_e} directly yields the recursive bound $\textnormal{dist}^2(x_{k+1}, \cX_{\star}) \leq \textnormal{dist}^2(x_k, \cX_{\star}) - t_k^2$. Moreover, by rearranging the terms in \eqref{eq:d_conv_e}, we obtain the expression
        \begin{align*}
            \norm{x_{k+1} - x_{\star}}_2^2
            \leq \parens{1 - \frac{t_k^2}{\norm{x_k - x_{\star}}_2^2}} \norm{x_k - x_{\star}}_2^2.
        \end{align*}
        While equivalent, the latter form makes it clear that the distance to the solution $x_\star$ decreases at a \emph{linear rate}. Not only is the rate linear, but it \emph{keeps improving} with each iteration, even if the radius $t_k \equiv t > 0$ is kept constant. Since in view of \eqref{eq:d_conv_e} the sequence $\{\norm{x_k - x_{\star}}_2\}_{k\geq 0}$ is strictly decreasing, the contraction factor $1 - \nicefrac{t^2}{\norm{x_k - x_{\star}}_2^2}$ also decreases with $k$, leading to progressively faster convergence. This behavior justifies calling the rate \emph{super-accelerated}. Moreover, the convergence rate is completely \emph{independent of the problem's condition number}.
        \newline
        Perhaps unexpectedly, the result holds without assuming smoothness or strong convexity--even without requiring milder alternatives such as the Polyak-Łojasiewicz condition, which are typically required to ensure linear rates \citep{karimi2016linear}.}
                    
        \item \textbf{Finite Convergence.}
        From the preceding point, it follows that if $\sum_{k=0}^{K-1} t_k^2 \geq \textnormal{dist}^2(x_0, \cX_{\star})$, then $x_K\in\cX_{\star}$.
        
        \textnormal{\color{gray} A direct consequence is that with a constant radius $t_k \equiv t > 0$, \algnamesmall{BPM} converges to the \textit{exact optimum} in a \textit{finite number of steps}: $K = \ceil{\nicefrac{\textnormal{dist}^2(x_0, \cX_{\star})}{t^2}}$. In other words, the super-accelerated decrease guaranteed by \eqref{eq:d_conv_e} continues to improve until the iterates reach the solution exactly, at which point the distance to $x_\star$ becomes exactly $0$. This sharply contrasts with all other optimization methods, including \algnamesmall{PPM}, which only converge asymptotically. \algnamesmall{BPM} can thus be seen as the first ``direct'' method of optimization.}
        
        \item \textbf{Super-Accelerated Linear Convergence in Function Values.}
        For any $k\geq 0$, the function value decreases according to
        \begin{align*}
            f(x_{k+1}) \leq f(x_k) - t_k \frac{f(x_{k+1}) - f_{\star}}{\norm{x_{k+1} - x_{\star}}},
        \end{align*}
        and hence
        \begin{eqnarray}\label{eq:f_conv_e}
            f(x_{k+1}) - f_{\star} \leq \parens{1 + \frac{t_k}{\norm{x_{k+1} - x_{\star}}_2}}^{-1} \parens{f(x_k) - f_{\star}}.
        \end{eqnarray}
        \textnormal{\color{gray} The result establishes a linear convergence rate for function value suboptimality, again without requiring strong convexity, differentiability, or finite-valuedness (thus covering constrained problems).}
        
        \item \textbf{Gradient Convergence.}
        If $f$ is differentiable, then $\norm{\nabla f(x_{k+1})}_2 \leq \norm{\nabla f(x_k)}_2$ for all $k\geq 0$, and
        \begin{eqnarray*}
            f(x_{k+1}) \leq f(x_k) - t_k \norm{\nabla f(x_{k+1})}.
        \end{eqnarray*}
        Therefore,
        \begin{eqnarray}\label{eq:grad_conv_e}
            \sum_{k=0}^{K-1} \parens{\frac{t_k}{\sum_{j=0}^{K-1} t_j} \norm{\nabla f(x_{k+1})}_2}
            \leq \frac{f(x_0) - f_{\star}}{\sum_{k=0}^{K-1} t_k}.
        \end{eqnarray}
        \textnormal{\color{gray} The bound in \eqref{eq:grad_conv_e} simplifies considerably when the radius is constant, i.e., $t_k \equiv t$. In that case, it becomes $\frac{1}{K} \sum_{k=0}^{K-1} \norm{\nabla f(x_{k+1})}_2 \leq \frac{f(x_0) - f_{\star}}{K t}$, which implies an $\cO(\nicefrac{1}{Kt})$ convergence rate for the average gradient norm.}
    \end{enumerate}
\end{theorem}
Collectively, \algnamesmall{BPM} exhibits striking and mathematically appealing convergence properties. To the best of our knowledge, it is the \emph{first direct optimization method}, in the sense that it solves the problem in a \emph{finite number of steps}. While direct methods have long been known in linear algebra (Gaussian elimination being a classic example), no comparable approach has previously existed in the optimization context.
But of course, this comes with a catch: the strength of these guarantees rests on the ability to compute the broximal operator; that is, ability to minimize the original objective function $f$ over a ball constraint. This is itself a potentially nontrivial optimization task and, in general, is difficult. This is especially true in nonconvex and high-dimensional regimes, where closed-form solutions are rarely available. As such, \algnamesmall{BPM} is best viewed not as a practical algorithm, but rather as an \emph{idealized framework}--a theoretical scaffold for designing and understanding a family of methods that approximate its behavior.
In this light, \algnamesmall{BPM} serves as a \emph{conceptual blueprint}, defining the outer boundaries of what is possible under geometric constraints. Yet, ``conceptual'' does not mean ``irrelevant''. As we explain in \Cref{rem:norm_gd} and further demonstrate in \Cref{sec:non_euclidean}, practical algorithms, particularly those in the \algnamesmall{Muon} family, can be viewed as \emph{approximations} of this ideal. Indeed, they can be interpreted as approximate broximal updates applied to surrogate subproblems. In doing so, these methods achieve tractability by trading exactness in the subproblem solution for computational feasibility.

\begin{remark}\label{rem:norm_gd}
    \Cref{thm:bpm_e} allows the radii $t_k$ to be arbitrarily large, and choosing a sufficiently large initial radius $t_0$ leads to convergence in a single step. However, such a strategy is clearly impractical. In real-world applications, we must rely on approximations to the original objective (as discussed in \Cref{sec:non_euclidean}), typically based on information at $x_k$. Once we introduce such approximations, it becomes necessary to impose upper bounds on $t_k$ to ensure that the model we optimize remains a faithful proxy for the true function $f$ on $\cB_2(x_k, t_k)$.

    This need for stepsize control is illustrated in \citet[Theorem F.4]{gruntkowska2025ball}, which analyzes the behavior of \algnamesmall{BPM} when applied to a linear approximation of $f$ at the current iterate, i.e.,
    \begin{align}\label{eq:lin_bpm0}
        x_{k+1} = \BProxSub{t_k}{f_k}{x_k},
    \end{align}
    where $f_k(z) \eqdef f(x_k) + \inp{\nabla f(x_k)}{z - x_k}$ (see \eqref{eq:muon_scion_det}).
    The authors show that when $f: \R^d \mapsto \R \cup \{+\infty\}$ is differentiable and convex, the iterates of \eqref{eq:lin_bpm0} (which, in this setting, are equivalent to normalized Gradient Descent--see \Cref{sec:muon_intro} and \Cref{sec:lin_bpm}) run with a sequence of radii $\{t_k\}_{k\geq0}$ such that
    \begin{align}\label{eq:asobnf}
        0 < t_k \leq \frac{\inp{\nabla f(x_k)}{x_k - x_\star}}{\norm{\nabla f(x_k)}_2}
    \end{align}
    satisfy
    \begin{eqnarray*}
        \norm{x_{k+1} - x_{\star}}_2^2 \leq \norm{x_k-x_{\star}}_2^2 - t_k^2.
    \end{eqnarray*}
    
    Condition \eqref{eq:asobnf} is satisfied, for instance, by setting $t_k = \nicefrac{(f(x_k) - f_\star)}{\norm{\nabla f(x_k)}_2}$, in which case the algorithm becomes equivalent to Gradient Descent with Polyak stepsize \citep{polyak1987introduction}.
        
    The result above can be derived by tracing the original \algnamesmall{BPM} proof \citep[Theorem 8.1]{gruntkowska2025ball} and modifying it to apply to $f_k$ rather than $f$. The guarantee closely mirrors the guarantee in \eqref{eq:d_conv_e}, with one critical distinction: here, the \emph{stepsizes must be bounded}. As a result, by introducing approximation, we gain implementability, but forfeit the arbitrarily fast convergence rate enjoyed by the idealized \algnamesmall{BPM}.

    This point of view offers a new understanding of why stepsize restrictions naturally arise in optimization algorithms. In practice, we do not work with perfect models of the objective function. Instead, we optimize certain approximations of $f$, which requires more conservative stepsizes to ensure that these approximations remain sufficiently accurate. In contrast, when operating with the exact model, one can recover remarkably strong guarantees, as shown in \Cref{thm:bpm_e}.
\end{remark}

\section{Beyond Euclidean Geometry}\label{sec:non_euclidean}

Euclidean geometry has long served as the standard framework for analyzing optimization algorithms, for example by relying on the classical $L$-smoothness assumption\footnote{The function $f$ is $L$--smooth if $\norm{\nabla f(x) - \nabla f(y)}_{\star} \leq L \norm{x - y}$ for all $x, y \in \cS$.} to facilitate convergence analysis. Yet, in the world of deep learning, where structure varies across layers and dimensions, it often fails to reflect the true nature of the problem. Recent success stories \citep{bernstein2024old, bernstein2018signsgd, jordan2024muon, pethick2025training, riabinin2025gluon} strongly suggest that stepping beyond the Euclidean framework can yield substantial benefits. In this section, we first provide a brief overview of how non-Euclidean geometries naturally emerge in modern optimization algorithms and how this understanding paves the way for a more general variant of the Broximal Point Method. We then present our main contributions: we formally introduce our Non-Euclidean Broximal Point Method, establish its theoretical properties, and explain why it offers additional advantages over the Euclidean version--benefits that extend beyond mere generality.

\subsection{Generalizing BPM: Lessons from Practice}\label{sec:muon_intro}

Although not originally designed with this perspective, algorithms such as \algnamesmall{Muon} \citep{jordan2024muon} can be interpreted as following one key principle: instead of directly optimizing the complex objective function $f$, they minimize a simplified \emph{model} of it---typically a linear approximation---within a ball defined by a suitable norm. This leads to update rules of the form\footnote{In practice, updates are applied \emph{layer-wise}, rather than to the full parameter vector as a whole. Specifically, the network parameter vector $x$ represents a collection of matrices $\mX^i \in \R^{m_i \times n_i}$, each corresponding to one of the $p$ layers $i = 1, \ldots, p$. The full parameter vector is $\mX = [\mX^1, \ldots, \mX^p] \in \cS$, where $\cS \eqdef \bigotimes_{i=1}^p \R^{m_i \times n_i}$. Each layer $i$ is associated with its own norm $\norm{\cdot}_{(i)}$, and the update rule in~\eqref{eq:muon_scion} is applied independently to each group, so that the algorithm iterates $\mX^i_{k+1} = \mX^i_k + t^i_k \lmo{\cB^i(\boldsymbol{0}, 1)}{\mG^i_k}$ for all $i = 1, \dots, p$, where $\mG^i_k$ is the momentum term for layer $i$ and $\cB^i(\mX, t)\eqdef\{\mZ \in \R^{m_i \times n_i} : \norm{\mX-\mZ}_{(i)} \leq t\}$ is the corresponding norm ball--for more details, see \citet{riabinin2025gluon}.}
\begin{align}\label{eq:muon_scion}
    x_{k+1} = x_k + t_k \lmo{\cB(0,1)}{g_k},
\end{align}
where $g_k$ represents a momentum term accumulating stochastic gradient information, $\cB(x,t)\eqdef\{z\in \cS : \norm{z-x} \leq t\}$ for some (possibly non-Euclidean) norm, and
\begin{align*}
    \lmo{\cB(x,t)}{g} \eqdef \argmin \limits_{z \in \cB(x,t)} \inp{g}{z}
\end{align*}
is the \emph{linear minimization oracle} (LMO) that outputs the minimizer of a certain linear function over a ball constraint.

The power of the above framework lies in its inherent flexibility: the geometry of the ball is dictated by the choice of the norm, which can be tailored to reflect the underlying structure of the parameter space. For instance, in deep networks, layer-specific operator norms can be used to better capture anisotropy across layers, leading to more effective updates and improved training dynamics compared to traditional methods \citep{liu2025muon, pethick2025training, riabinin2025gluon}.

At a higher level, the general update rule~\eqref{eq:muon_scion} defines a large family of optimization algorithms, parameterized by the choice of the norm defining the ball constraint.
In matrix spaces, for instance, selecting the \emph{operator norm} $\norm{\mA}_{\alpha \to \beta} \eqdef \sup _{\norm{z}_\alpha=1} \norm{\mA z}_\beta$ and setting $\norm{\cdot} = \norm{\cdot}_{2\to2}$, the LMO becomes $\lmo{\cB(0, 1)}{\mG_k} = - \mU_k \mV_k^T$, where $\mG_k = \mU_k \textnormal{diag}(\sigma_k) \mV_k^T$ is the (reduced) singular value decomposition of the momentum matrix~$\mG_k$. In this case,~\eqref{eq:muon_scion} becomes $\mX_{k+1} = \mX_k - t_k \mU_k \mV_k^T$, which is precisely the update rule applied to hidden layers by \algnamesmall{Muon}/\algnamesmall{Scion}.\footnote{\algnamefootnote{Muon} is an optimizer specifically designed for hidden layers; the first and last layers are optimized using a different method, such as \algnamefootnote{Adam(W)} \citep{jordan2024muon}.}

Although \algnamesmall{Muon} provides the primary motivation for extending our framework beyond the Euclidean setting, it is by no means the only geometry choice of interest. For instance, when $\cS = \R^d$, selecting the $\ell_1$ norm recovers the update rule of Coordinate Descent (\algnamesmall{CD}) \citep{nesterov2012efficiency, wright2015coordinate, richtarik2011iteration}, using the $\ell_2$ norm corresponds to normalized Gradient Descent (\algnamesmall{$\|$GD$\|$}) \citep{nesterov1984minimization, hazan2015beyond, Momentum_Norm_SGD_2020}, and choosing the $\ell_\infty$ norm yields Sign Gradient Descent (\algnamesmall{SignGD}) \citep{riedmiller1993direct, bernstein2018signsgd, safaryan2021stochastic}. Additional connections and discussion can be found in \Cref{sec:lit_rev} and \Cref{sec:lin_bpm}.

In the context of this paper, a particularly illustrative case arises when momentum is disabled and full gradients are used instead. Under this setting, \eqref{eq:muon_scion} reduces to
\begin{align}\label{eq:muon_scion_det}
    x_{k+1} &= x_k + t_k \lmo{\cB(0,1)}{\nabla f(x_k)}
    = \argmin \limits_{z \in \cB(x_k,t_k)} \inp{\nabla f(x_k)}{z} \nonumber \\
    &= \argmin \limits_{z \in \cB(x_k,t_k)} \brac{f(x_k) + \inp{\nabla f(x_k)}{z - x_k}} \nonumber \\
    &= \argmin \limits_{z \in \cS} \brac{f_k(z): {\color{myblue} \norm{z - x_k} \leq t_k}},
\end{align}
where $f_k(z) \eqdef f(x_k) + \inp{\nabla f(x_k)}{z - x_k}$ is the linearization of $f$ at the current iterate $x_k$.

This perspective reveals a deep structural similarity between~\eqref{eq:muon_scion}, \eqref{eq:muon_scion_det}, and the update rule of \ref{eq:bpm_e}, with two key differences:
\begin{enumerate}[label=(\roman*)]
    \item \ref{eq:bpm_e} minimizes the true objective $f$, whereas~\eqref{eq:muon_scion_det} minimizes its linear approximation $f_k$,
    \item \ref{eq:bpm_e} performs the minimization over a Euclidean ball, whereas~\eqref{eq:muon_scion_det} does so over a general norm ball.
\end{enumerate}
In this work, we do not pursue the model-based aspect of point~(i). Instead, we focus on (ii), investigating the consequences of generalizing the ball geometry in \algnamesmall{BPM} beyond Euclidean norms.

\subsection{Non-Euclidean BPM}\label{sec:non_euclidean_bpm}

Building on these insights, we propose a direct generalization of the Broximal Point Method by simply replacing the Euclidean norm in the constraint with an arbitrary norm $\norm{\cdot}$. This leads to the Non-Euclidean Broximal Point Method, with the following update rule:
\begin{align}\label{eq:bpm_ne}
    x_{k+1} = \argmin \limits_{z \in \cS} \brac{f(z): {\color{myblue} \norm{z - x_k} \leq t_k}}. \tag{Non-Euclidean \algnamesmall{BPM}}
\end{align}

Like its Euclidean predecessor, this method is inherently \emph{conceptual} in nature--the broximal operator may be difficult to compute exactly. Nonetheless, it \emph{can} be made practical through various approximations, for example:
\begin{itemize}
    \item \textbf{Solving subproblems approximately:} Use an iterative solver to approximately minimize the original objective $f$ over the ball $\cB(x_k, t_k)$,
    \item \textbf{Solving approximate subproblems:} Replace $f$ with some simpler \emph{model}, such as its linearization $f_k$, and solve the resulting (often tractable) trust region subproblem instead.
\end{itemize}
As we have already seen in \Cref{sec:muon_intro}, the latter approach admits closed-form solutions in certain settings. Crucially, the effectiveness of this strategy is not confined to the Euclidean setting, and its power becomes most apparent in the non-Euclidean contexts. Indeed, this is precisely the mechanism underlying the updates of \algnamesmall{Muon} and \algnamesmall{Scion}, as shown in~\eqref{eq:muon_scion_det}. The \algnamesmall{Muon} family thus represents just one concrete instantiation of the broad spectrum of approximations that can be captured and analyzed within our framework.

In a broader context, \algnamesmall{Muon} and \algnamesmall{Scion} can be interpreted as instances of a non-Euclidean trust region method \citep{conn2000trust} (see \Cref{sec:lit_rev}) applied to a linear model of the objective function, as first noted by \citet{kovalev2025understanding}. Meanwhile, both Euclidean and Non-Euclidean \algnamesmall{BPM} represent the \emph{idealized} trust region method, applied directly to the actual objective $f$ itself.

\subsubsection{Theoretical Guarantees}\label{sec:non_euclidean_bpm_theory}

Interestingly, Non-Euclidean \algnamesmall{BPM} preserves most (though not all) of the convergence guarantees established for its Euclidean counterpart. The following theorem, the main contribution of our work, demonstrates linear convergence in terms of function value suboptimality and monotonic decline of gradient norms:

\begin{restatable}{theorem}{THMMAIN}\label{thm:conv-nbpm}
    Assume that $f: \cS \mapsto \R \cup \{+\infty\}$ is proper, closed and convex, $\cX_\star \neq \emptyset$, and let $\{x_k\}_{k\geq0}$ be the iterates of Non-Euclidean \algnamesmall{BPM} run with any sequence of positive radii $\{t_k\}_{k\geq0}$, where $x_0\in {\rm dom} f$. Then 
    \begin{enumerate}[label=(\roman*)]
        \item If $\cX_{\star}\cap \cB(x_k, t_k)\neq\emptyset$, then $x_{k+1} \in \cX_{\star}$.\label{pt:n_opt}
        
        \item \label{pt:t_dist_nbpm}
        If $\cX_{\star}\cap \cB(x_k, t_k)=\emptyset$, then $\norm{x_{k+1} - x_k} = t_k$.
        
        \item \label{pt:1step_nbpm}
        For any $k\geq 0$, 
        \begin{eqnarray*}
            f(x_{k+1}) - f_{\star} \leq \parens{1 + \frac{t_k}{\norm{x_{k+1} - x_{\star}}}}^{-1} \parens{f(x_k) - f_{\star}}.
        \end{eqnarray*}

        \item If $f$ is differentiable, then $\norm{\nabla f(x_{k+1})}_{\star} \leq \norm{\nabla f(x_k)}_{\star}$ for all $k\geq 0$, and 
        \begin{eqnarray*}
            \squeeze \sum_{k=0}^{K-1} \parens{\frac{t_k}{\sum_{k=0}^{K-1} t_k} \norm{\nabla f(x_{k+1})}_{\star}}
            \leq \frac{f(x_0) - f_{\star}}{\sum_{k=0}^{K-1} t_k}.
        \end{eqnarray*}
    \end{enumerate}
\end{restatable}
\Cref{thm:conv-nbpm} presents a direct analogue to the convergence result for function values and gradient norms in the Euclidean case (\eqref{eq:f_conv_e} and \eqref{eq:grad_conv_e}). Consequently, all the observations made in \Cref{thm:bpm_e} remain applicable in this context. Similar guarantees can also be derived for a certain class of non-convex functions, as discussed in \citet{gruntkowska2025ball}. Importantly, the results in \Cref{thm:conv-nbpm} are not a mere generalization of those in \Cref{thm:bpm_e}; with a suitable choice of norm, they can in fact yield stronger convergence guarantees, as we discuss in \Cref{sec:norm_prec}.

As reiterated throughout this work, Non-Euclidean \algnamesmall{BPM} may not be directly implementable, and one may need to resort to approximations. Several prior works have explored the method arising by replacing the true objective with a linear model, as in \eqref{eq:muon_scion_det} (often motivated by the \algnamesmall{Muon} framework but allowing for arbitrary LMOs, not just those arising from the choice $\norm{\cdot} = \norm{\cdot}_{2\to2}$).
In particular, the recent work of \citet{kovalev2025understanding} interprets \eqref{eq:muon_scion_det} as a trust region method,\footnote{A trust region interpretation in the Euclidean case was previously discussed by \citet{gruntkowska2025ball}.} where the regions are defined by norm balls $\cB(x_k,t_k)$. In the star-convex and $L$-smooth setting, the author proves a convergence guarantee of the form
\begin{align*}
    f(x_k) - f_\star \leq \varepsilon \qquad\textnormal{in}\qquad \cO(\nicefrac{LD^2}{\varepsilon}) \textnormal{ steps},
\end{align*}
where $D>0$ is the diameter of ${\rm dom} f$ (\citet[Corollary 7]{kovalev2025understanding}). This matches the classical Gradient Descent rate for smooth convex problems (up to logarithmic factors) \citep{nesterov2018lectures}. The analysis requires the stepsizes to be $t_k = \cO(\nicefrac{\varepsilon}{LD})$, further reinforcing the point made in \Cref{rem:norm_gd}: replacing the exact model with an approximation necessarily leads to (potentially very conservative) stepsize bounds. Compared to our guarantees, the results of \citet{kovalev2025understanding} are significantly weaker. They offer no finite-time convergence, no superlinear behavior, not even a linear rate, rely on strong assumptions and use radii dependent on the target accuracy. 

This highlights the fundamental trade-off between exactness and implementability. The method studied by \citet{kovalev2025understanding} is effectively a trust region algorithm applied to a linearized model of $f$, that is, a first-order approximation of the idealized method we analyze. In contrast, by working directly with the exact model, Non-Euclidean \algnamesmall{BPM} enjoys substantially stronger convergence guarantees under minimal assumptions.

Comparing \Cref{thm:bpm_e} and \Cref{thm:conv-nbpm}, one important caveat arises: unlike in the Euclidean case, where the distance to the minimizer is guaranteed to decrease monotonically and linearly for any positive radii sequence $\{t_k\}_{k\geq0}$ (as established in \eqref{eq:d_conv_e}), this property fails to hold in general normed spaces. In fact, even for convex objective functions, the distance to the solution set $\cX_{\star}$ may \emph{increase} when moving from $x_k$ to $x_{k+1}$, unless the radius $t_k$ is sufficiently large (in the extreme case when $t_0 \geq \textnormal{dist}(x_0, \cX_{\star})$, the method reaches the minimizer in a single step, just as in the Euclidean setting). A simple example illustrating this behavior is provided in~\Cref{fig:example}.
Having said that, a distance-based convergence guarantee analogous to~\eqref{eq:d_conv_e} can still be recovered, provided that the norm is induced by an inner product (see \Cref{thm:conv-ebrox} in \Cref{sec:ellipsoids}).

\begin{figure}[ht]
    \includegraphics[width=0.5\linewidth, valign=c]{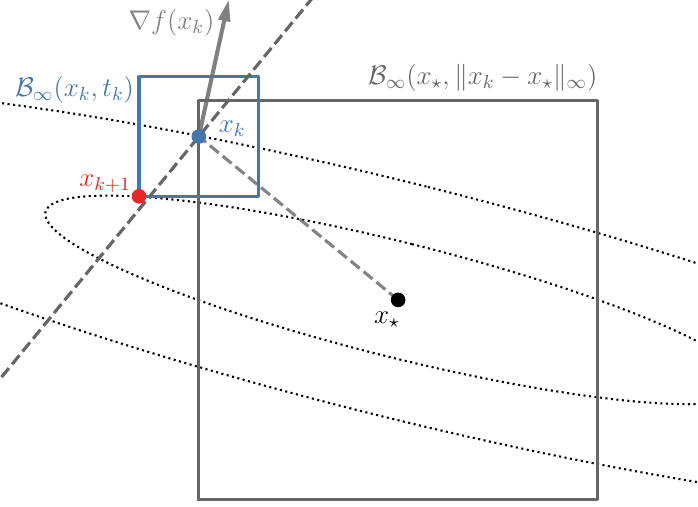}
    \hfill
    \begin{minipage}{\dimexpr 0.45\linewidth-\columnsep}
        \caption{\textbf{In the non-Euclidean case, the distance to $\cX_{\star}$ need not decrease.} Let $\norm{\cdot} = \norm{\cdot}_{\infty}$ be the infinity norm, and consider a simple two-dimensional example illustrating the behavior of Non-Euclidean \algnamesmall{BPM} when applied to a convex quadratic objective. The level sets of the function $f$ are depicted as {\color{gray} gray} dotted ellipses. At each iteration, the algorithm minimizes $f$ over the $\ell_{\infty}$ ball ${\color{myblue} \cB_{\infty}(x_k, t_k)}$ centered at the current iterate {\color{myblue} $x_k$}, and moves to the minimizer within this region, denoted {\color{myred2} $x_{k+1}$}.
        We observe that $x_{k+1}$ does not lie within the $\ell_\infty$ ball centered at $x_\star$ with radius $\norm{x_k - x_\star}_\infty$, indicating that the distance to the solution increases.
        This example highlights that in non-Euclidean geometries, monotonic progress toward the solution cannot be guaranteed, even for convex problems.}
        \label{fig:example}
    \end{minipage}
\end{figure}

\subsubsection{Why the Norm Matters: Geometric Preconditioning}\label{sec:norm_prec}

A natural question remains: beyond increased generality, does the result in \Cref{thm:conv-nbpm} offer any additional theoretical advantages over its Euclidean counterpart in \Cref{thm:bpm_e}? In prior work on LMO-type algorithms with arbitrary norms \citep{pethick2025training, kovalev2025understanding, riabinin2025gluon}, non-Euclidean geometry has proven beneficial due to reliance on smoothness assumptions. In such settings, aligning the geometry with the problem structure can yield significantly smaller smoothness constants and, consequently, improved convergence rates.
Our analysis departs from that framework by sidestepping any reliance on a smoothness model. As a result, one might think that norm choice plays a less significant role. This is not the case. Even without smoothness-based reasoning, selecting an appropriate norm is still a critical factor, though for a different reason.

To illustrate this, let us focus on the first step of the algorithm and imagine that we can freely choose the geometry of the norm ball $\cB(x_0, t_0)$. To ensure a fair comparison, suppose that all balls under consideration have the same fixed volume $V>0$. The goal, naturally, is to select a geometry so that one step of the algorithm brings us as close to $x_\star$ as possible, ideally reaching $x_\star$ itself. This is always geometrically feasible: by ``stretching'' the ball in the right direction, we can ensure that $x_\star \in \cB(x_0, t_0)$ without changing the volume.
As a concrete example, consider the norm $\norm{x}_{\mX} \eqdef \sqrt{x^\top \mX x}$, where $\mX \in \R^{d \times d}$ is a symmetric positive definite matrix. The associated norm ball of radius $t_0$ centered at $x_0$ is the $d$-dimensional ellipsoid
\begin{align*}
    \cB_{\mX}(x_0, t_0) \eqdef \brac{z\in \R^d : \norm{z-x_0}_{\mX} \leq t_0}.
\end{align*}
If we are free to ``play'' with $\cB_{\mX}(x_0, t_0)$ (subject to the fixed volume constraint), we can construct $\mX$ so that the ball ``touches'' the solution, effectively solving the problem in one iteration. To this end, set $t_0 = \norm{x_0 - x_\star}_{\mX}$, which imposes the constraint $\norm{x_0 - x_\star}_{\mX}^d \det(\mX)^{-1/2} = \frac{V}{\vol(\cB_2(0,1))}$, where $\vol(\cB_2(0,1))$ is the volume of the $d$-dimensional unit Euclidean ball. This equation always admits a solution. For instance, one can choose
\begin{align*}
    \mX = \parens{\frac{V}{\norm{x_0 - x_\star}_{\mX}^d \vol(\cB_2(0,1))}}^{\frac{2}{d-1}} \mP + \mP^\perp,
\end{align*}
where $\mP = \frac{(x_0 - x_\star) (x_0 - x_\star)^\top}{\norm{x_0 - x_\star}_2^2}$ (see \Cref{thm:vol_arg}). This choice guarantees that $x_\star \in \cB(x_0, t_0)$, satisfying both the geometric and volume constraints.

If the direction $x_0-x_\star$ were known beforehand, one could exploit this information to align the geometry accordingly, leading to extremely fast convergence. Of course, perfect knowledge of the solution is not available in practice (if it were, there would be no need for an iterative algorithm at all). Nevertheless, even partial prior information about the problem can guide the design of more effective transformations, resulting in improved practical performance.

The above construction can be interpreted as a new form of \emph{preconditioning} \citep{hestenes1952methods, benzi2002preconditioning}, with the matrix $\mX$ playing the role of the \emph{preconditioner} (see \Cref{sec:lit_rev}). More generally, norm selection goes beyond this classical approach, as it is not restricted to Mahalanobis-type metrics. In this sense, it serves as \emph{geometric preconditioning}--a nonparametric, non-Euclidean analogue of preconditioning. This geometric flexibility may help explain the empirical success of the methods discussed in this work: by adapting to the underlying geometry, they implicitly tend to yield better-conditioned optimization problems.

The idea of using knowledge about the distance to the solution is somewhat reminiscent of the recent breakthrough in optimization--\algnamesmall{D-adaptation} \citep{defazio2023learning}, recipient of the 2023 ICML Outstanding Paper Award. That method incorporates an estimate of the distance $D = \norm{x_0-x_\star}$ into its algorithmic design by iteratively constructing lower bounds on $D$, which are then used to guide adaptive stepsize selection.
Although the two approaches differ significantly---\citet{defazio2023learning} operate under a standard gradient oracle and cannot achieve our type of results---they share a common principle: distance matters. In our setting, this is reflected in how norm choice implicitly encodes both geometry and stepsize.
While the settings and mechanisms are distinct, both highlight the fundamental role that distance to the optimum plays in efficient optimization.

Last but not least, the norm choice can be highly influential when analyzing the linearized variant of Non-Euclidean \algnamesmall{BPM}. Indeed, if an analogue of the result in \Cref{rem:norm_gd} were to hold in the non-Euclidean setting, the upper bound on the stepsize would be strongly dictated by the norm appearing in the denominator.

\section{Related Work}\label{sec:lit_rev}

\paragraph{Ball Oracles.}

Ball oracles---subroutines that minimize a function over a norm ball---have been central to several recent advances in optimization. Prior to the method proposed by \citet{gruntkowska2025ball}, which is the focus of the main part of this paper, \citet{carmon2020acceleration} introduced a framework for acceleration based on this primitive, achieving near-optimal complexity guarantees. The approach has since been adapted to various settings: \citet{carmon2021thinking} and \citet{asi2021stochastic} applied it to the problem of minimizing the maximum loss, while \citet{carmon2023resqueing} and \citet{jambulapati2024closing} developed parallel methods. Subsequently, \citet{carmon2022optimal} tightened convergence bounds by improving logarithmic dependencies, and \citet{adil2024convex} extended the framework to $\ell_p$ norm balls.

\paragraph{Preconditioned Gradient Methods.}

Standard stochastic gradient methods can converge very slowly when applied to ill-conditioned problems. A classical remedy is \emph{preconditioning}, a well-established technique in optimization and numerical linear algebra that accelerates convergence by transforming the problem geometry to reduce ill-conditioning. In the context of first-order methods, this is achieved by modifying the standard update rule by introducing a sequence of symmetric positive definite matrices $\{\mX_k\}_{k\geq0}$, known as \emph{preconditioners}, which scale the gradient direction. This corresponds to performing steepest descent in the norm $\norm{x}_{\mX_k} \eqdef \sqrt{x^\top \mX_k x}$, yielding the update
\begin{align*}
    x_{k+1} = x_k - \gamma_k \mX_k^{-1} g_k,
\end{align*}
where $g_k$ is a (stochastic) gradient estimate and $\gamma_k > 0$ is a stepsize.
This idea dates back to early works in numerical optimization \citep{hestenes1952methods}, and has since been extensively studied in both deterministic and stochastic settings. The choice of $\mX_k$ plays a critical role: fixed preconditioners based on curvature approximations are often used in convex problems, while adaptive or data-driven variants are more common in machine learning. Prominent examples in the latter category include \algnamesmall{AdaGrad} \citep{duchi2011adaptive}, \algnamesmall{Adam} \citep{kingma2014adam}, \algnamesmall{BFGS} \citep{gower2016stochastic, gower2018accelerated, kovalev2020fast}, and \algnamesmall{Shampoo} \citep{gupta2018shampoo}. Preconditioned methods also form the basis for second-order and quasi-Newton algorithms \citep{gill1972quasi, lewis2013nonsmooth, gower2017randomized, bottou2018optimization, kovalev2019stochastic, islamov2023distributed}. All of these approaches can be interpreted as variants of Stochastic Gradient Descent performed in a dynamically rescaled coordinate system, leading to improved robustness and faster convergence in practice.

\paragraph{Coordinate Descent Methods.}

Coordinate Descent (\algnamesmall{CD}) algorithms have a long and rich history \citep{southwell1940relaxation, powell1973search, zhiquan1993error, shalevshwartz2009Stochastic, richtarik2011iteration, richtarik2012parallel, nesterov2012efficiency, richtarik2013distributed, fercoq2013accelerated, wright2015coordinate, qu2016coordinate, nesterov2017efficiency, sigma_k}. At a high level, these methods iteratively optimize the objective function by fixing most components of the parameter vector and updating a selected subset only. By focusing on one coordinate (or one \emph{block} of coordinates) at a time, they decompose high-dimensional problems into a sequence of simpler, lower-dimensional subproblems.
In their basic form, \algnamesmall{CD} methods applied to problem \eqref{eq:problem} with $\cS = \R^d$ proceed by selecting a subset (called \emph{block}) $b_k \subseteq [d]$, and updating the corresponding coordinates $x^{b_k}_k \in \R^{|b_k|}$ according to
\begin{align*}
    x^{b_k}_{k+1} = x^{b_k}_k - \gamma_k g^{b_k}_k,
\end{align*}
where $\gamma_k>0$ is the stepsize and $g^{b_k}_k \in \R^{|b_k|}$ is a suitably chosen descent direction for the lower-dimensional subproblem. The remaining coordinates of $x_k$ are left unchanged.
This strategy often leads to substantial reductions in per-iteration computational and memory costs, making \algnamesmall{CD} methods both scalable and easy to implement. Furthermore, their amenability to parallelization and ability to exploit problem structure have made Block Coordinate Descent (\algnamesmall{BCD}) particularly attractive for large-scale optimization problems \citep{richtarik2011iteration, richtarik2012parallel, beck2013convergence, nutini2017let}.

\paragraph{Sign Descent Methods.}

Sign-based optimization methods (\algnamesmall{SignSGD}) \citep{bernstein2018signsgd} originated from efforts to simplify and accelerate large-scale optimization, and have gained traction in machine learning due to their low communication overhead and surprisingly strong empirical performance in neural network training. The idea, popularized by algorithms like \algnamesmall{RPROP} \citep{riedmiller1993direct}, is to replace the full gradient with its element-wise sign, retaining directional information while discarding magnitudes. This results in a highly compressed gradient representation, making the approach very attractive for distributed or large-scale settings.
Methods in this family perform updates of the form
\begin{align*}
    x_{k+1} = x_k - \gamma_k \sign\parens{g_k},
\end{align*}
where $g_k$ is a (stochastic) gradient estimate and $\sign(\cdot)$ is applied component-wise. 
Interest in sign-based methods surged in the past decade, partly due to their close connection to adaptive optimizers such as \algnamesmall{Adam} \citep{kingma2014adam}. In fact, when exponential moving averages are disabled, \algnamesmall{Adam} reduces to \algnamesmall{SignSGD} \citep{balles2018dissecting, balles2020geometry}. Sign descent methods have since been the subject of extensive analysis, with recent works investigating their convergence properties, limitations, and interpretations \citep{karimireddy2019error, safaryan2021stochastic, kunstner2023noise, bernstein2024old}.

\paragraph{LMO-based Optimizers.}

A classical family of optimization methods based on Linear Minimization Oracles (LMOs) are the Frank-Wolfe (\algnamesmall{FW}) algorithms, also known as Conditional Gradient methods \citep{frank1965algorithm, jaggi2013revisiting}. Originally designed for constrained optimization, \algnamesmall{FW} algorithms replace costly projection or proximal steps with linear minimization over the feasible set, making them particularly attractive in high-dimensional problems where projections are expensive or intractable.

In recent years, LMO-based optimizers have been adapted to the deep learning context. Algorithms of this type iterate by minimizing surrogate models (e.g., linearizations of the loss) over non-Euclidean norm balls. This strategy seeks to better capture layer-wise structure and directional anisotropy in the loss landscape through the careful selection of norms, and has led to strong empirical performance in training deep neural networks \citep{liu2025muon, pethick2025training, riabinin2025gluon, shah2025practical, therien2025muloco, tveit2025muon}. Notable examples of such optimizers include \algnamesmall{Muon} \citep{jordan2024muon} and \algnamesmall{Scion} \citep{pethick2025training}. The \algnamesmall{Muon} optimizer was initially introduced as an effective empirical method for optimizing hidden layers (with other optimizers, typically \algnamesmall{AdamW}, applied to the first and last layers). Later, \citet{pethick2025training} formally connected such updates to the \algnamesmall{FW} framework and proposed \algnamesmall{Scion}, which employs LMO-based updates across all layers, using layer-specific norms. Subsequent theoretical works \citep{kovalev2025understanding, li2025note} have analyzed simplified global variants of these optimizers under standard $L$-smoothness assumption. Building on this, \citet{riabinin2025gluon} advanced the theory by establishing convergence guarantees under a more realistic layer-wise $(L_0,L_1)$-smoothness assumption, which better reflects the practical layer-wise nature of these methods.

While \algnamesmall{Muon} and \algnamesmall{Scion} predominantly rely on spectral norms, other choices are possible. In particular, considering $\ell_p$ norms recovers the previously discussed coordinate descent methods (for $p=1$) and sign descent methods (for $p=\infty$), which we elaborate on in \Cref{sec:lin_bpm}.

\paragraph{Trust Region Methods.}

Trust region methods are a well-established family of optimization algorithms that minimize an objective function $f$ by iteratively solving simpler surrogate problems within a localized neighborhood of the current iterate. At each step, these algorithms construct a model $m_k(x)$, typically a quadratic approximation of $f$, that is assumed to be reliable within a specified region, known as the \emph{trust region}, around the current point $x_k$ \citep{conn2000trust}. This region is most commonly defined as a norm ball $\cB(x_k, t_k) \eqdef \brac{z\in\R^d: \norm{z-x_k} \leq t_k}$, where $t_k$ is the \emph{trust region radius}, though more sophisticated variants may employ ellipsoidal or box-shaped regions to better align with the problem's geometry. The next iterate $x_{k+1}$ is obtained by minimizing the model $m_k(x_k)$ over this region. After each step, the quality of the approximation is evaluated and the trust region radius is adjusted accordingly.

In this context, LMO-based optimizers such as \algnamesmall{Muon} and \algnamesmall{Scion} can be viewed as non-Euclidean trust region methods applied to a linear and stochastic approximation of $f$, whereas \algnamesmall{BPM} represents an idealized trust region method that operates directly on the true objective.

\section{Conclusion}\label{sec:conclusion}

In this work, inspired by recent breakthroughs in the design of optimizers capable of iteratively minimizing a linear approximation of the objective function over balls defined via arbitrary norms, we focus on extending the recently proposed Broximal Point Method (\algnamesmall{BPM}) \citep{gruntkowska2025ball} to the non-Euclidean setting. 

The resulting Non-Euclidean \algnamesmall{BPM} offers an idealized meta-algorithm with deep links to a growing family of geometry-aware optimizers. While practical methods like \algnamesmall{Muon} and \algnamesmall{Scion} operate on linear surrogates of the objective and rely on implementable LMOs, our method replaces these approximations with exact subproblem solutions, revealing the structural essence that underlies their success. In doing so, it provides a conceptual blueprint and makes a step towards clarifying the role of geometry in shaping optimization trajectories and global convergence properties.

Naturally, some of the remarkable guarantees achievable in the Euclidean case cannot be extended to arbitrary norm geometries. We explicitly demonstrate why such results may break down. We also leave out important practical aspects such as stochastic gradients, momentum, and non-convexity--all of which are central to modern optimization methods \citep{jordan2024muon, pethick2025training, riabinin2025gluon}. Incorporating these practically relevant components is an important direction for future work; here, however, our emphasis is on the clean deterministic setting.

Nonetheless, the broader picture remains: Non-Euclidean \algnamesmall{BPM} enriches our theoretical toolkit, offering a platform for designing and analyzing new algorithms. We view this work as a step towards a deeper theoretical foundation for the emerging class of geometry-aware optimization algorithms, and as a source of simple yet elegant inspiration for future developments.

\section*{Acknowledgements}

The research reported in this publication was supported by funding from King Abdullah University of Science and Technology (KAUST): i) KAUST Baseline Research Scheme, ii) CRG Grant ORFS-CRG12-2024-6460, and iii) Center of Excellence for Generative AI, under award number 5940.

\bibliographystyle{plainnat}
\bibliography{references}

\newpage

\tableofcontents

\newpage

\appendix

\section*{APPENDIX}

\section{Useful Facts and Lemmas}

In this section, we collect key definitions and present several fundamental results needed for the main convergence proof in \Cref{sec:main_proof}. We begin by formalizing our notation.

Throughout, we let ${\rm dom} f \eqdef \{x \in \cS : f(x)<+\infty\}$ denote the domain of a function $f: \cS \to \R \cup \{+\infty\}$. Given a set $\cC \subseteq \cS$, we denote its boundary, interior, and relative interior by $\bdry(\cC)$, $\interior(\cC)$, and $\ri(\cC)$, respectively.

For any nonempty, closed, and convex set $\cC \subseteq \cS$, we define its \emph{indicator function} $\delta_{\cC}: \cS \to \R \cup \{+\infty\}$ by
\begin{align*}
    \delta_{\cC}(z) \eqdef
    \begin{cases}
        0 & \text{if } z\in \cC, \\
        +\infty & \text{otherwise}.
    \end{cases}
\end{align*}
This function is proper, closed, and convex.

\begin{definition}[Subdifferential]
    Let $f: \cS \mapsto \R \cup \{+\infty\}$ be proper and let $x \in {\rm dom}(f)$. The \emph{subdifferential} of $f$ at $x$, denoted as $\partial f(x)$, is the set of vectors $g\in\cS$ such that
    \begin{align*}
        f(y) \geq f(x) + \inp{g}{y-x} \qquad \forall y\in\cS.
    \end{align*}
    The elements of $\partial f(x)$ are called the \emph{subgradients} of $f$ at $x$.
\end{definition}

\begin{fact}\label{fact:conv_ball_norm}
    In any normed space $(\cS, \norm{\cdot})$, the closed ball $\cB(x,t) \eqdef \brac{z: \norm{z-x} \leq t}$ is convex.
\end{fact}

\begin{fact}[Subdifferential of indicator function]\label{lemma:subfiff_id}
    Let $\cC \subseteq \cS$ be a nonempty convex set. The subdifferential of an indicator function of $\cC$ at a point $y \in \cC$ is
    \begin{align*}
        \partial \delta_{\cC}(y) = \cN_{\cC}(y) \eqdef \brac{g\in\cS : \inp{g}{z - y} \leq 0 \,\forall z \in \cC},
    \end{align*}
    where $\cN_{\cC}(y)$ is the normal cone of $\cC$ at $y$.
\end{fact}

\begin{fact}[Normal cone of a norm ball]\label{fact:cone_ball}
    Let $\norm{\cdot}$ be any norm on $\cS$. The normal cone of a ball $\cB(x,t) = \brac{z \in\cS: \norm{z-x} \leq t}$ at a point $y \in \cB(x,t)$ is
    \begin{align*}
        \cN_{\cB(x,t)}(y) = \brac{g \in\cS: t \norm{g}_{\star} \leq \inp{g}{y - x}},
    \end{align*}
    where $\norm{\cdot}_{\star}$ is the dual norm of $\norm{\cdot}$.
\end{fact}

\begin{proof}
    Let $y \in \cB(x,t)$. Then
    \begin{align*}
        g \in \partial \delta_{\cB(x,t)}(y) \qquad&\overset{\eqref{lemma:subfiff_id}}{\iff}\qquad
        \inp{g}{z - y} \leq 0 \quad\forall z \in \cB(x,t) \\\qquad&\iff\qquad
        \inp{g}{z} \leq \inp{g}{y} \quad\forall z \in \cB(x,t) \\\qquad&\iff\qquad
        \sup_{z: \norm{z-x} \leq t} \inp{g}{z} \leq \inp{g}{y} \\\qquad&\iff\qquad
        \sup_{z: \norm{\frac{z - x}{t}} \leq 1} \inp{g}{\frac{z - x}{t}} \leq \inp{g}{\frac{y - x}{t}} \\\qquad&\iff\qquad
        \sup_{w: \norm{w} \leq 1} \inp{g}{w} \leq \inp{g}{\frac{y - x}{t}} \\\qquad&\iff\qquad
        \norm{g}_{\star} \leq \frac{\inp{g}{y - x}}{t},
    \end{align*}
    which finishes the proof.
\end{proof}

\begin{lemma}\label{lemma:n_cone_partial}
    Let $u \in \cB(x, t)$ and consider any $g\in\cN_{\cB(x, t)}(u)$. Then
    \begin{align*}
        \inp{g}{u - x} = \norm{g}_{\star} \norm{u - x}.
    \end{align*}
\end{lemma}
\begin{proof}
    By \Cref{fact:cone_ball}, we know that
    \begin{align*}
        \cN_{\cB(x, t)}(u) = \brac{g \in \cS: t \norm{g}_{\star} \leq \inp{g}{u - x}}.
    \end{align*}
    Therefore, using Cauchy-Schwarz inequality, for any $g\in\cN_{\cB(x, t)}(u)$ we have
    \begin{align*}
        t \norm{g}_{\star} \leq \inp{g}{u - x} \leq \norm{g}_{\star} \norm{u - x} \leq t \norm{g}_{\star}.
    \end{align*}
    Hence all inequalities must be equalities, which implies the claimed identity.
\end{proof}

\begin{theorem}\label{thm:n2ndlp}
    Let $f: \cS \mapsto \R \cup \{+\infty\}$ be proper, closed and convex. Choose $x \in {\rm dom} f$ and $u \in \argmin_{z \in \cB(x, t)} f(z)$, where $t>0$. Then there exists $g\in \cN_{\cB(x, t)}(u)$ such that $-g \in \partial f(u)$, i.e.,
    \begin{align*}
        f(y) - f(u) \geq \inp{g}{u-y}
    \end{align*}
    for all $y \in \cS$.
\end{theorem}
\begin{proof}
    The proof follows similar ideas to those used in the proof of \citet[Theorem D.2]{gruntkowska2025ball}.
    First, we show that $\ri(\cB(x, t)) \cap \ri(\dom(f)) \neq \emptyset$. This is immediate if $x \in \ri(\dom(f))$. Suppose instead that $x \not\in \ri(\dom(f))$. Then $\overline{\ri(\dom(f))} = \overline{\dom(f)} \ni x$, so there exists a sequence $\{z_k\}_{k\geq0} \subset \ri(\dom(f))$ such that $z_k \to x$ as $k\to\infty$. Since $t>0$ and $x \in \cB(x, t)$, there exists~$K\geq 0$ such that $z_k \in \ri(\cB(x, t))$ for all $k \geq K$.
    It follows that $\ri(\cB(x, t)) \cap \ri(\dom(f)) \neq \emptyset$. Consequently, since both $f$ and $\cB(x, t)$ are convex, we may apply \citet[Proposition 6.19]{bauschke2011convex} to conclude that $0 \in \sri(\cB(x, t) - \dom(f))$. Then, by \citet[Proposition 27.8]{bauschke2011convex}, there exists $g\in \cN_{\cB(x, t)}(u)$ such that $-g \in \partial f(u)$. By the definition of the subdifferential, this implies that
    \begin{align*}
        f(y) - f(u) \geq \inp{g}{u-y}
    \end{align*}
    for all $y \in \cS$.
\end{proof}

The next result generalizes the statement of \citet[Theorem D.1]{gruntkowska2025ball}.

\begin{theorem}\label{thm:1stconv}
    Let $f: \cS \mapsto \R \cup \{+\infty\}$ be proper, closed and convex and $\cC \subseteq \cS$ be a non-empty closed and convex set. Then
    \begin{enumerate}[label=(\roman*)]
        \item $\argmin_{z\in\cC} f(z) \neq \emptyset$. Moreover, if $\cC \cap \cX_{\star} \neq \emptyset$, then $\argmin_{z\in\cC} f(z)$ is a non-empty subset of $\cX_{\star}$. \label{pt:conv_nonempty}
        \item If $\cC \cap \cX_{\star} = \emptyset$, then $\argmin_{z\in\cC} f(z)$ lies on the boundary of $\cC$. Moreover, if $\cC$ is strictly convex, the minimizer is unique. \label{pt:conv_singleton_bdry}
    \end{enumerate}
\end{theorem}

\begin{proof}
    \begin{enumerate}[label=(\roman*)]
        \item The function $f$ is proper, closed and convex, and $\cC$ is closed. Hence, by the Weierstrass theorem, $f$ is lower bounded and attains its minimum over $\cC$. Hence $\argmin_{z\in\cC} f(z) \neq \emptyset$. If $\cC \cap \cX_{\star} \neq \emptyset$, then clearly $\argmin_{z \in \cC} f(z) \subseteq \cX_{\star}$ is non-empty.
        
        \item Let $z_\star \in \argmin_{z\in\cC} f(z)$. 
        Then $z_\star$ is a minimizer of the function
        \begin{align*}
            \psi(z) \eqdef f(z) + \delta_{\cC}(z).
        \end{align*}
        Suppose for contradiction that $z_\star \in \interior (\cC)$ and consider the line segment connecting~$z_\star$ and any global minimizer $x_\star$ of $f$. By assumption, $x_\star \notin \cC$, and hence the line segment must intersect $\bdry (\cC)$ at some point $z_\lambda = \lambda z_\star + \rbrac{1 - \lambda}x_\star$, where $\lambda \in (0, 1)$. Since $f$ is convex, we have 
        \begin{align*}
            f(z_\lambda) \leq \rbrac{1 - \lambda}f(x_\star) + \lambda f(z_\star) < f(z_\star),
        \end{align*}
        where the last inequality is strict because $x_\star \in \cX_{\star}$, $z_\star \in \cC$ and $\cX_{\star}\cap \cC = \emptyset$. But then $\psi(z_\lambda) < \psi(z_\star)$, contradicting the optimality of $z_\star$. Therefore, any minimizer must lie on $\bdry(\cC)$.
        
        Now, assume that $\cC$ is strictly convex but $\argmin_{z\in\cC} f(z)$ is not a singleton. Then there exist two distinct minimizers $z_{\star,1}, z_{\star,2} \in \argmin_{z\in\cC} f(z)$, and by the previous argument, both must lie on $\bdry(\cC)$. But then, due to the convexity of $f$, all points on the line segment connecting $z_{\star, 1}$ and $z_{\star, 2}$ are also minimizers of $\psi(z)$. Since $\cC$ is strictly convex, this contradicts the earlier conclusion that no minimizer of $\psi(z)$ can lie in $\interior(\cC)$. We conclude that $\argmin_{z\in\cC} f(z)$ must be a singleton.
    \end{enumerate}
\end{proof}

\section{Proof of the Main Theorem}\label{sec:main_proof}

\THMMAIN*
\begin{proof}
    \begin{enumerate}[label=(\roman*)]
        \item This follows from \Cref{fact:conv_ball_norm} and \Cref{thm:1stconv}~\ref{pt:conv_nonempty}.
        
        \item This follows from \Cref{fact:conv_ball_norm} and \Cref{thm:1stconv}~\ref{pt:conv_singleton_bdry}. 
    
        \item Consider some iteration $k$ such that $x_{k+1} \not\in \cX_{\star}$ (otherwise, the problem is solved in $1$ step). \Cref{thm:n2ndlp} with $y = x = x_k$ gives
        \begin{align*}
            f(x_k) - f(x_{k+1}) \geq \inp{g}{x_{k+1}-x_k}
        \end{align*}
        for some $g\in \cN_{\cB(x_k, t_k)}(x_{k+1})$, and hence by \Cref{lemma:n_cone_partial},
        \begin{eqnarray}\label{eq:untbsvyab}
            f(x_{k+1}) - f_{\star} &\leq& f(x_k) - f_{\star} - \inp{g}{x_{k+1}-x_k}
            = f(x_k) - f_{\star} - \norm{g}_{\star} \norm{x_{k+1}-x_k} \nonumber \\
            &\overset{\ref{pt:t_dist_nbpm}}{=}& f(x_k) - f_{\star} - t_k \norm{g}_{\star}.
        \end{eqnarray}
        By \Cref{thm:n2ndlp} and Cauchy-Schwarz inequality, we also have
        \begin{align*}
            f(x_{k+1}) - f_{\star} \leq - \inp{g}{x_{k+1}-x_{\star}}
            \leq \norm{g}_{\star} \norm{x_{k+1}-x_{\star}}.
        \end{align*}
        Since $x_{k+1} \not\in \cX_{\star}$, we can rearrange this inequality, obtaining
        \begin{eqnarray}\label{eq:iasbfrfa}
            \parens{f(x_{k+1}) - f_{\star}} \frac{t_k}{\norm{x_{k+1} - x_{\star}}} \leq t_k \norm{g}_{\star}.
        \end{eqnarray}
        Applying the bound \eqref{eq:iasbfrfa} in \eqref{eq:untbsvyab} gives
        \begin{eqnarray*}
            f(x_{k+1}) - f_{\star} \leq f(x_k) - f_{\star} - \parens{f(x_{k+1}) - f_{\star}} \frac{t_k}{\norm{x_{k+1} - x_{\star}}},
        \end{eqnarray*}
        and rearranging the terms, we obtain
        \begin{eqnarray*}
            f(x_{k+1}) - f_{\star} \leq \parens{1 + \frac{t_k}{\norm{x_{k+1} - x_{\star}}}}^{-1} \parens{f(x_k) - f_{\star}}
        \end{eqnarray*}
        as required.

        \item Again, suppose that $x_{k+1} \not\in \cX_{\star}$. Since $f$ is differentiable, $\partial f(u) = \{\nabla f(u)\}$ for all $u\in\cS$, and hence, according to \Cref{thm:n2ndlp}, there exists $g\in \cN_{\cB(x_k, t_k)}(x_{k+1})$ such that $-g \in \partial f(x_{k+1}) = \{\nabla f(x_{k+1})\}$. Then, \Cref{lemma:n_cone_partial} says that
        \begin{align}\label{eq:uasesldnv}
            \inp{- \nabla f(x_{k+1})}{x_{k+1} - x_k} = \norm{\nabla f(x_{k+1})}_{\star} \norm{x_{k+1} - x_k}
            \overset{\ref{pt:t_dist_nbpm}}{=} t_k \norm{\nabla f(x_{k+1})}_{\star}.
        \end{align}
        Now, convexity and Cauchy-Schwarz inequality give
        \begin{align*}
            f(x_{k+1}) - f(x_k) \geq \inp{\nabla f(x_k)}{x_{k+1} - x_k}
            \geq - \norm{\nabla f(x_k)}_{\star} \norm{x_{k+1} - x_k}
            \overset{\ref{pt:t_dist_nbpm}}{=} - t_k \norm{\nabla f(x_k)}_{\star}.
        \end{align*}
        Rearranging the terms and using convexity again, we obtain
        \begin{eqnarray*}
            t_k \norm{\nabla f(x_k)}_{\star} &\geq& f(x_k) - f(x_{k+1}) \\
            &\geq& \inp{\nabla f(x_{k+1})}{x_k - x_{k+1}} \\
            &\overset{\eqref{eq:uasesldnv}}{=}& t_k \norm{\nabla f(x_{k+1})}_{\star},
        \end{eqnarray*}
        which proves the first part of the claim.
        To prove the second part, we again use convexity to obtain
        \begin{eqnarray*}
            f(x_{k+1}) \leq f(x_k) - \inp{\nabla f(x_{k+1})}{x_k - x_{k+1}}
            \overset{\eqref{eq:uasesldnv}}{=} f(x_k) - t_k \norm{\nabla f(x_{k+1})}_{\star}.
        \end{eqnarray*}
        Rearranging the terms and summing over the first $K$ iterations yields
        \begin{eqnarray*}
            \sum_{k=0}^{K-1} \parens{t_k \norm{\nabla f(x_{k+1})}_{\star}}
            \leq \sum_{k=0}^{K-1} \parens{f(x_k) - f(x_{k+1})}
            = f(x_0) - f(x_K)
            \leq f(x_0) - f_{\star}.
        \end{eqnarray*}
        Dividing both sides of the inequality above by $\sum_{k=0}^{K-1} t_k$ proves the claim.
    \end{enumerate}
\end{proof}

One can establish a result analogous to that in \ref{pt:1step_nbpm} by employing a proof strategy similar to that in \citet[Theorem 26]{carmon2020acceleration}. Specifically, under the assumptions of \Cref{thm:bpm_e}, the authors demonstrate that
\begin{align*}
    f(x_{k+1}) - f_{\star} \leq \parens{1-\frac{t_k}{R}} \parens{f(x_k) - f_{\star}},
\end{align*}
where $R$ is a constant such that $\norm{x_0 - x_\star}_2 \leq R$, where $x_{\star}\in\cX_{\star}$. This result is specific to the Euclidean setting. However, the same proof technique can be adapted to obtain a bound more closely aligned with that in \ref{pt:1step_nbpm} in the non-Euclidean case, as formalized in the following theorem.

\begin{theorem}\label{thm:con_cbrox}
    Let the assumptions of \Cref{thm:conv-nbpm} hold and let $\{x_k\}_{k\geq0}$ be the iterates of \ref{eq:bpm_ne} run with any sequence of positive radii $\{t_k\}_{k\geq0}$, where $x_0\in {\rm dom} f$. Then
    \begin{align*}
        f(x_{k+1}) - f_{\star} \leq \parens{1-\frac{t_k}{\norm{x_k - x_\star}}} \parens{f(x_k) - f_{\star}}.
    \end{align*}
\end{theorem}
\begin{proof}
    Consider some iteration $k$ such that $x_{k+1} \not\in \cX_{\star}$ and let $z$ be a point where the line segment $[x_k, x_\star]$ intersects $\bdry(\cB(x_k,t_k))$. Then $z \in {\rm dom} f \cap \cB(x_k,t_k)$, so $f(z) \geq f(x_{k+1})$ since $x_{k+1}$ is a minimizer of $f$ over $\cB(x_k,t_k)$. Therefore, convexity of $f$ gives
    \begin{align*}
        f(x_{k+1}) &\leq f(z)
        = f\parens{\parens{1-\frac{\norm{x_k - z}}{\norm{x_k - x_\star}}} x_k + \frac{\norm{x_k - z}}{\norm{x_k - x_\star}} x_\star} \\
        &\leq \parens{1-\frac{\norm{x_k - z}}{\norm{x_k - x_\star}}} f(x_k) + \frac{\norm{x_k - z}}{\norm{x_k - x_\star}} f_\star.
    \end{align*}
    Rearranging and using the fact that $\norm{x_k - z} = t_k$, we obtain
    \begin{align*}
        f(x_{k+1}) - f_{\star} \leq \parens{1-\frac{t_k}{\norm{x_k - x_\star}}} \parens{f(x_k) - f_{\star}}
    \end{align*}
    as needed.
\end{proof}

\subsection{Convergence of Distances for Norms Induced by an Inner Product}\label{sec:ellipsoids}

In this section, we establish an additional convergence result for the distances between iterates and the minimizer when the underlying norm is induced by an inner product. Specifically, we consider the norm $\norm{x}_{\mX} \eqdef \sqrt{x^\top \mX x}$, where $\mX \in \R^{d \times d}$ is a symmetric positive definite matrix. The corresponding norm balls are $d$-dimensional ellipsoids, and we denote the ball of radius~$t$ centered at~$x$ by
\begin{align*}
    \cB_{\mX}(x, t) = \brac{z\in \R^d : \norm{z-x}_{\mX} \leq t}.
\end{align*}
In this setting, Non-Euclidean \algnamesmall{BPM} with the norm choice $\norm{\cdot} = \norm{\cdot}_{\mX}$ iterates
\begin{align}\label{eq:bpm_ne_el}
    x_{k+1} = \argmin \limits_{z \in \R^d} \brac{f(z): {\color{myblue} \norm{z - x_k}_{\mX} \leq t_k}} = \argmin \limits_{z \in \cB_{\mX}(x_k, t_k)} f(z),
\end{align}
where $\{t_k\}_{k \geq 0}$ is a sequence of positive radii.

As in previous section, we begin by presenting some facts and lemmas that will be useful in the main proof (\Cref{thm:conv-ebrox}).

\begin{fact}[Theorem 3.40 of \citet{beck2017first}]
    \label{fact:2:sum-rule-subdiff}
    Let $f_i: \R^d \mapsto \R \cup \cbrac{+\infty}$, $i \in [n]$, be proper convex functions such that $\cap_{i=1}^n \ri({\rm dom} f_i) \neq \emptyset$.
    Then
    \begin{align*}
        \partial \rbrac{\sum_{i=1}^{n} f_i}(x) = \sum_{i=1}^{n}\partial f_i(x)
    \end{align*}
    for any $x \in \R^d$.
\end{fact}

\begin{fact}[Normal cone of the indicator function of an ellipsoid]\label{fact:subdiff_id_ell}
    The normal cone of $\cB_{\mX}(x, t)$ is
    \begin{align*}
        \cN_{\cB_{\mX}(x, t)}(y) = 
        \begin{cases}
            \R_{\geq0} \mX (y-x) & \norm{x-y}_{\mX} = t, \\
            \{0\} & \norm{x-y}_{\mX} < t, \\
            \emptyset & \norm{x-y}_{\mX} > t.
        \end{cases}
    \end{align*}
\end{fact}

The next result is a consequence of \Cref{thm:n2ndlp}.

\begin{corollary}\label{cor:2ndebrox}
    Let $f: \R^d \mapsto \R \cup \{+\infty\}$ be proper, closed and convex. Choose $x \in {\rm dom} f$ and $u \in \argmin_{z \in \cB_{\mX}(x, t)} f(z)$, where $t>0$.
    Then, there exists $c_t(x) \geq 0$ such that
    \begin{enumerate}[label=(\roman*)]
        \item $c_t(x) \mX (x - u) \in \partial f(u)$,
        \item $f(y) - f(u) \geq c_t(x)\inp{\mX (x - u)}{y - u}$ for all $y \in \R^d$.
    \end{enumerate}
\end{corollary}

\begin{proof}
    Suppose first that $\cB_{\mX}(x, t) \cap \cX_{\star} \neq \emptyset$. Then, by \Cref{thm:1stconv}, we have $u\in\cX_{\star}$, which implies that $0 \in \partial f(u)$. Therefore, statement $(i)$ holds with $c_t(x) = 0$. Since $u$ is a global minimizer of $f$, it follows that $f(y) \geq f(u)$ for all $y \in \R^d$, so statement $(ii)$ also holds.

    Now, suppose instead that $\cB_{\mX}(x, t) \cap \cX_{\star} = \emptyset$. In this case, \Cref{thm:n2ndlp} guarantees the existence of a vector $g\in \cN_{\cB_{\mX}(x, t)}(u)$ such that $-g \in \partial f(u)$. Moreover, \Cref{thm:1stconv} ensures that $\norm{x-u}_{\mX} = t$. The conclusion then follows directly from \Cref{fact:subdiff_id_ell}.
\end{proof}

\begin{theorem}\label{thm:conv-ebrox}
    Assume that $f: \R^d \mapsto \R \cup \{+\infty\}$ is proper, closed and convex, and let $\{x_k\}_{k\geq0}$ be the iterates of Non-Euclidean \algnamesmall{BPM} with $\norm{\cdot} = \norm{\cdot}_{\mX}$ run with any sequence of positive radii $\{t_k\}_{k\geq0}$, where $x_0\in {\rm dom} f$.
    If $\cX_{\star}\cap \cB_{\mX}(x_k, t_k)=\emptyset$, then $\norm{x_{k+1} - x_k}_{\mX} = t_k$. Moreover, for any $x_{\star}\in\cX_{\star}$, we have
    \begin{align*}
        \norm{x_{k+1} - x_{\star}}_{\mX}^2 \leq \norm{x_k - x_{\star}}_{\mX}^2 - t_k^2
    \end{align*}
    and
    \begin{align*}
        \textnormal{dist}^2(x_{k+1}, \cX_{\star}) \leq \textnormal{dist}^2(x_k, \cX_{\star}) - t_k^2.
    \end{align*}
    Hence, if $\sum_{k=0}^{K-1} t_k^2 \geq \textnormal{dist}^2(x_0, \cX_{\star})$, then $x_K\in\cX_{\star}$.
\end{theorem}

\begin{proof}
    Suppose that $\cX_{\star}\cap \cB_{\mX}(x_k, t_k)=\emptyset$. The fact that $\norm{x_{k+1} - x_k}_{\mX} = t_k$ is an immediate consequence of \Cref{thm:1stconv}.  We now turn to the proof of the first inequality. Applying \Cref{cor:2ndebrox} with $x = x_k$ and $y = x_\star$, we obtain
    \begin{eqnarray*}
        f(x_{k+1}) - f_{\star} \leq c_{t_k}(x_k) \inp{\mX (x_k-x_{k+1})}{x_{k+1} - x_{\star}}
    \end{eqnarray*}
    for some $c_{t_k}(x_k)\geq 0$. In fact, the inequality is strict. Indeed, if $c_{t_k}(x_k)=0$, then $f(x_{k+1}) - f_{\star} \leq 0$, implying $x_{k+1} \in \cX_{\star}$, which contradicts the assumption that $\cX_{\star}\cap B_{t_k}(x_k)=\emptyset$. Therefore, $c_{t_k}(x_k) > 0$, and we may divide both sides of the inequality to get
    \begin{eqnarray*}
        0 &\leq& \frac{f(x_{k+1}) - f_{\star}}{c_{t_k}(x_k)} \leq \inp{\mX (x_k-x_{k+1})}{x_{k+1} - x_{\star}} \\
        &=& \frac{1}{2} \parens{\norm{x_k-x_{\star}}_{\mX}^2 - \norm{x_{k+1} - x_{\star}}_{\mX}^2 - \norm{x_k-x_{k+1}}_{\mX}^2} \\
        &=& \frac{1}{2} \parens{\norm{x_k-x_{\star}}_{\mX}^2 - \norm{x_{k+1} - x_{\star}}_{\mX}^2 - t_k^2},
    \end{eqnarray*}
    and the first inequality follows. Since this holds for any $x_{\star}\in\cX_{\star}$, it also holds for the point in $\cX_{\star}$ that is closest to $x_k$. Noting that $\textnormal{dist}^2(x_{k+1}, \cX_{\star}) \leq \norm{x_{k+1} - x_{\star}}_{\mX}^2$, we obtain the recursive inequality in terms of distances. Finally, the fact that if $\sum_{k=0}^{K-1} t_k^2 \geq \textnormal{dist}^2(x_0, \cX_{\star})$ then $x_K\in\cX_{\star}$ follows immediately from this result.
\end{proof}

\subsection{Norm Design Under Fixed Volume Constraints}

To support the claim made in the \Cref{sec:norm_prec} that it is always geometrically feasible to construct a norm ball of fixed volume that contains the solution, we now provide a formal theorem statement and proof. In particular, we show that for any pair of points $x_0, x_\star$ and any target volume $V > 0$, one can construct a Mahalanobis norm $\norm{\cdot}_{\mX}$ such that the corresponding ellipsoid of radius $t_0 = \norm{x_0 - x_\star}_{\mX}$ has volume exactly $V$. This guarantees that $x_\star$ lies on the boundary of the ellipsoid, and demonstrates that geometry can always be adapted so that the initial ball includes the solution.

\begin{theorem}\label{thm:vol_arg}
    Let $x_0, x_\star \in \mathbb{R}^d$ be distinct points, and let $V > 0$. Then, there exists a symmetric positive definite matrix $\mX \in \mathbb{R}^{d \times d}$ such that the volume of the $d$-dimensional ellipsoid $\cB_{\mX}(x_0, t_0) \eqdef \brac{z\in \R^d : \norm{z-x_0}_{\mX} \leq t_0}$, where $t_0 = \norm{x_0 - x_\star}_{\mX}$, is $\vol(\cB_{\mX}(x_0, t_0))=V$.
\end{theorem}

\begin{proof}
    Define the rank-one projector $\mP \eqdef \frac{(x_0 - x_\star) (x_0 - x_\star)^\top}{\norm{x_0 - x_\star}_2^2}$ and its orthogonal complement $\mP^\perp \eqdef \mI - \mP$. Consider the matrix
    \begin{align*}
        \mX \eqdef c_1 \mP + c_2 \mP^\perp,
    \end{align*}
    where $c_1, c_2 > 0$ are scalars to be determined. Recall that
    \begin{align}\label{eq:v}
        V = t_0^d \det(\mX)^{-1/2} \vol(\cB_2(0,1))
        = \norm{x_0 - x_\star}_{\mX}^d \det(\mX)^{-1/2} \vol(\cB_2(0,1)),
    \end{align}
    where $\vol(\cB_2(0,1))$ is the volume of the $d$-dimensional unit Euclidean ball.
    Here, $\norm{x_0 - x_\star}^2_{\mX} = (x_0 - x_\star)^\top \mX (x_0 - x_\star) = c_1 \norm{x_0 - x_\star}^2$, and
    \begin{align*}
        \det(\mX) &= \det\parens{c_2 \parens{\mI + \frac{c_1-c_2}{c_2} \mP}} \\
        &= c_2^d \parens{1 + \frac{c_1-c_2}{c_2} \frac{(x_0 - x_\star)^\top (x_0 - x_\star)}{\norm{x_0 - x_\star}_2^2}} \\
        &= c_1 c_2^{d - 1},
    \end{align*}
    where we used the fact that $\det(\mI + u v^\top) = 1 + u^\top v$. Substituting this into \eqref{eq:v}, we have
    \begin{align*}
        V = (\sqrt{c_1} \norm{x_0 - x_\star}_2)^d (c_1 c_2^{d-1})^{-1/2} \vol(\cB_2(0,1))
        = \norm{x_0 - x_\star}_2^d c_1^{(d - 1)/2} c_2^{-(d - 1)/2} \vol(\cB_2(0,1)),
    \end{align*}
    and hence
    \begin{align*}
        \parens{\frac{c_1}{c_2}}^{(d - 1)/2} = \frac{V}{\norm{x_0 - x_\star}_2^d \vol(\cB_2(0,1))}.
    \end{align*}
    This always has a positive solution for $c_1 > 0$ given any choice of $c_2 > 0$. For example, setting $c_2 = 1$ yields a unique $c_1 > 0$ satisfying the equation, giving
    \begin{align*}
        \mX = c_1 \mP + c_2 \mP^\perp = \parens{\frac{V}{\norm{x_0 - x_\star}_{\mX}^d \vol(\cB_2(0,1))}}^{\frac{2}{d-1}} \mP + \mP^\perp.
    \end{align*}
    Noting that $\mX = c_1 \mP + c_2 \mP^\perp \succ \boldsymbol{0}$ finishes the proof.
\end{proof}

\section{Linearized BPM -- Special Cases}\label{sec:lin_bpm}

We now examine the linearized form of the Non-Euclidean \algnamesmall{BPM} to illustrate how different norm choices give rise to several well-known algorithms. Recall from \eqref{eq:muon_scion_det} that the algorithm update rule can be written as
\begin{align}\label{eq:lin_bpm}
    x_{k+1} &= x_k + t_k \lmo{\cB(0,1)}{\nabla f(x_k)}
    = \argmin \limits_{z \in \cB(x_k,t_k)} \inp{\nabla f(x_k)}{z} \nonumber \\
    &= \argmin \limits_{z \in \cB(x_k,t_k)} \brac{f(x_k) + \inp{\nabla f(x_k)}{z - x_k}} \nonumber \\
    &= \argmin \limits_{z \in \cS} \brac{f_k(z): {\color{myblue} \norm{z - x_k} \leq t_k}},
\end{align}
where $f_k(z) \eqdef f(x_k) + \inp{\nabla f(x_k)}{z - x_k}$ is the linearization of $f$ at the current iterate $x_k$.

Different choices of the norm $\norm{\cdot}$ yield different LMO update rules. Below, we examine some special cases. Note that operations on vectors in $\R^d$, such as $\sign(x)$, $|x|$, $xy$ etc., are applied  component-wise and return a vector in $\R^d$.
\begin{enumerate}
    \item \textbf{$\ell_1$ norm.}
    Let $\cS = \R^d$. For the $\ell_1$ norm, the LMO is given by 
    \begin{align*}
        \lmo{\cB_1(0,t_k)}{y} = - t_k \sign\parens{[y]_{i_{\max}}} e_{i_{\max}},
    \end{align*}
    where $i_{\max} \in \argmax_{i\in[n]} |[y]_i|$, and Algorithm \eqref{eq:lin_bpm} iterates
    \begin{align}\label{eq:lin_bpm_l1}
        x_{k+1} = x_k - t_k \sign\parens{[\nabla f(x_k)]_{i_{\max}}} e_{i_{\max}},
    \end{align}
    recovering Coordinate Descent (\algnamesmall{CD}) \citep{wright2015coordinate} with the greedy Gauss-Southwell selection rule (i.e., choosing $i_k$ such that $|[\nabla f(x_k)]_i|$ is maximized) \citep{nutini2015coordinate}. 
    
    \item \textbf{$\ell_2$ norm.}
    Let $\cS = \R^d$. For the $\ell_2$ norm, we have 
    \begin{align*}
        \lmo{\cB_2(0,t_k)}{y} = - t_k \frac{y}{\norm{y}_2}.
    \end{align*}
    Thus, Algorithm \eqref{eq:lin_bpm} updates
    \begin{align*}
        x_{k+1} = x_k - \frac{t_k}{\norm{\nabla f(x_k)}_2} \nabla f(x_k),
    \end{align*}
    which corresponds to normalized Gradient Descent (\algnamesmall{$\|$GD$\|$}).

    \item \textbf{$\ell_\infty$ norm.}
    Let $\cS = \R^d$. For the $\ell_\infty$ norm, we obtain 
    \begin{align*}
        \lmo{\cB_\infty(0,t_k)}{y} = - t_k \sign\parens{y},
    \end{align*}
    and Algorithm \eqref{eq:lin_bpm} iterates
    \begin{align}\label{eq:lin_bpm_l_inf}
        x_{k+1} = x_k - t_k \sign\parens{\nabla f(x_k)}.
    \end{align}
    This recovers Sign Gradient Descent (\algnamesmall{SignGD}) \citep{riedmiller1993direct, bernstein2018signsgd}.

    \item \textbf{$\ell_p$ norm, $p\in(1,\infty)$.}
    Let $\cS = \R^d$. In general, for any $p\in(1,\infty)$, the LMO takes the form
    \begin{align*}
        \lmo{\cB_p(0,t_k)}{y} = - t_k \frac{\sign\parens{y} |y|^{q-1}}{\norm{y}_q^{q-1}},
    \end{align*}
    where $\frac{1}{p} + \frac{1}{q} =1$. Thus, Algorithm \eqref{eq:lin_bpm} iterates
    \begin{align*}
        x_{k+1} = x_k - \frac{t_k}{\norm{\nabla f(x_k)}_q^{q-1}} \sign\parens{\nabla f(x_k)} |\nabla f(x_k)|^{q-1},
    \end{align*}
    which interpolates between \eqref{eq:lin_bpm_l1} and \eqref{eq:lin_bpm_l_inf}.

    \item \textbf{Spectral norm.}
    Let $\cS = \R^{m\times n}$. Then, for $\norm{\cdot} = \norm{\cdot}_{2\to2}$, the LMO is
    \begin{align*}
        \lmo{\cB(\boldsymbol{0},t_k)}{\mY} = - t_k \mU \mV^T,
    \end{align*}
    where $\mY = \mU \textnormal{diag}(\sigma) \mV^T$ is the (reduced) singular value decomposition of~$\mY$. 
    Algorithm \eqref{eq:lin_bpm} iterates
    \begin{align*}
        \mX_{k+1} = \mX_k - t_k \mU_k \mV_k^T,
    \end{align*}
    (where $\mG_k = \mU_k \textnormal{diag}(\sigma_k) \mV_k^T$).
    This gives the update rule applied to hidden layers by \algnamesmall{Muon}/\algnamesmall{Scion} \citep{jordan2024muon, pethick2025training}.
\end{enumerate}

\end{document}